\documentclass{scrartcl}
\usepackage{amsmath}
\usepackage{amsfonts}
\usepackage{hyperref}
\usepackage{graphicx}
\newcommand{\bbbr}{\mathbb{R}}
\newcommand{\bbbn}{\mathbb{N}}
\newtheorem{theorem}{Theorem}
\newtheorem{definition}[theorem]{Definition}

\newtheorem{remark}[theorem]{Remark}
\newenvironment{proof}{Proof.}{$\Box$}
\allowdisplaybreaks
\newcommand{\treeroot}{\mathop{\operatorname{root}}\nolimits}
\newcommand{\sons}{\mathop{\operatorname{sons}}\nolimits}
\newcommand{\level}{\mathop{\operatorname{level}}\nolimits}
\newcommand{\Idx}{\mathcal{I}}
\newcommand{\Jdx}{\mathcal{J}}
\newcommand{\Kdx}{\mathcal{K}}
\newcommand{\ctI}{\mathcal{T}_{\Idx}}
\newcommand{\ctJ}{\mathcal{T}_{\Jdx}}
\newcommand{\ctK}{\mathcal{T}_{\Kdx}}
\newcommand{\lfI}{\mathcal{L}_{\Idx}}
\newcommand{\lfJ}{\mathcal{L}_{\Jdx}}
\newcommand{\lfK}{\mathcal{L}_{\Kdx}}
\newcommand{\ctIJ}{\mathcal{T}_{\Idx\times\Jdx}}
\newcommand{\ctJK}{\mathcal{T}_{\Jdx\times\Kdx}}
\newcommand{\ctIK}{\mathcal{T}_{\Idx\times\Kdx}}
\newcommand{\ctII}{\mathcal{T}_{\Idx\times\Idx}}
\newcommand{\lfIJ}{\mathcal{L}_{\Idx\times\Jdx}}
\newcommand{\lfJK}{\mathcal{L}_{\Jdx\times\Kdx}}
\newcommand{\lfIK}{\mathcal{L}_{\Idx\times\Kdx}}
\newcommand{\lfII}{\mathcal{L}_{\Idx\times\Idx}}
\newcommand{\lfaIJ}{\mathcal{L}_{\Idx\times\Jdx}^+}
\newcommand{\lfiIJ}{\mathcal{L}_{\Idx\times\Jdx}^-}
\newcommand{\lfaJK}{\mathcal{L}_{\Jdx\times\Kdx}^+}
\newcommand{\lfiJK}{\mathcal{L}_{\Jdx\times\Kdx}^-}

\newcommand{\ctIJK}{\mathcal{T}_{\Idx\times\Jdx\times\Kdx}}
\newcommand{\ctsub}[1]{\mathcal{T}_{#1}}
\newcommand{\lfsub}[1]{\mathcal{L}_{#1}}
\newcommand{\pI}{p_\Idx}
\newcommand{\pJ}{p_\Jdx}
\newcommand{\pIJ}{p_{\Idx\times\Jdx}}
\newcommand{\pJK}{p_{\Jdx\times\Kdx}}
\newcommand{\pIK}{p_{\Idx\times\Kdx}}
\newcommand{\Cqr}{C_{\mathrm{qr}}}
\newcommand{\Csv}{C_{\mathrm{sv}}}
\newcommand{\Csp}{C_{\mathrm{sp}}}
\newcommand{\Csn}{C_{\mathrm{sn}}}
\newcommand{\Cad}{C_{\mathrm{ad}}}
\newcommand{\Cmg}{C_{\mathrm{mg}}}
\newcommand{\Cmm}{C_{\mathrm{mm}}}
\newcommand{\Wev}{W_{\mathrm{ev}}}

\newcommand{\Wup}{W_{\mathrm{up}}}
\newcommand{\trunc}{\operatorname{trunc}\nolimits}
\newcommand{\blocktrunc}{\operatorname{blocktrunc}\nolimits}
\title{Hierarchical matrix arithmetic with accumulated updates}
\author{Steffen B\"orm}
\date{\today}
\begin{document}
\maketitle
\begin{abstract}
Hierarchical matrices can be used to construct efficient preconditioners
for partial differential and integral equations by taking advantage of
low-rank structures in triangular factorizations and inverses of the
corresponding stiffness matrices.

The setup phase of these preconditioners relies heavily on low-rank
updates that are responsible for a large part of the algorithm's total
run-time, particularly for matrices resulting from three-dimensional
problems.

This article presents a new algorithm that significantly reduces the
number of low-rank updates and can shorten the setup time by 50 percent
or more.
\end{abstract}

%
%
\section{Introduction}

Hierarchical matrices \cite{HA99,GRHA02,HA15} (frequently abbreviated
as $\mathcal{H}$-matrices) employ the special structure of integral
operators and solution operators arising in the context of elliptic
partial differential equations to approximate the corresponding
matrices efficiently.
The central idea is to exploit the low numerical ranks of suitably
chosen submatrices to obtain efficient factorized representations
that significantly reduce storage requirements and the computational
cost of evaluating the resulting matrix approximation.

Compared to similar approximation techniques like panel clustering
\cite{HANO89,SA00}, fast multipole algorithms \cite{RO85,GRRO87,GRRO91},
or the Ewald fast summation method \cite{EW20}, hierarchical matrices
offer a significant advantage:
it is possible to formulate algorithms for carrying out (approximate)
arithmetic operations like multiplication, inversion, or factorization
of hierarchical matrices that work in almost linear complexity.
These algorithms allow us to construct fairly robust and efficient
preconditioners both for partial differential equations and
integral equations.

Most of the required arithmetic operations can be reduced to the
matrix multiplication, i.e., the task of updating
$Z \gets Z + \alpha X Y$, where $X$, $Y$, and $Z$ are hierarchical
matrices and $\alpha$ is a scaling factor.
Once we have an efficient algorithm for the multiplication,
algorithms for the inversion, various triangular factorizations,
and even the approximation of matrix functions like the matrix
exponential can be derived easily
\cite{GR01a,GRHAKH02,GRLE04,GAHAKH00,GAHAKH02,BA08}.

The $\mathcal{H}$-matrix multiplication in turn can be reduced to
two basic operations:
the multiplication of an $\mathcal{H}$-matrix by a thin dense matrix,
equivalent to multiple parallel matrix-vector multiplications, and
low-rank updates of the form $Z \gets Z + A B^*$, where $A$ and $B$
are thin dense matrices with only a small number of columns.
Since the result $Z$ has to be an $\mathcal{H}$-matrix again, these
low-rank updates are always combined with an approximation step that
aims to reduce the rank of the result.
The corresponding rank-revealing factorizations (e.g., the singular
value decomposition) are responsible for a large part of the
computational work of the $\mathcal{H}$-matrix multiplication and,
consequently, also inversion and factorization.

The present paper investigates a modification of the standard
$\mathcal{H}$-matrix multiplication algorithm that draws upon
inspiration from the \emph{matrix backward transformation} employed
in the context of $\mathcal{H}^2$-matrices \cite{BO04a,BO10}:
instead of applying each low-rank update immediately to an
$\mathcal{H}$-matrix, multiple updates are \emph{accumulated} in
an auxiliary low-rank matrix, and this auxiliary matrix is
propagated as the algorithm traverses the hierarchical structure
underlying the $\mathcal{H}$-matrix.
Compared to the standard algorithm, this approach reduces the work
for low-rank updates from $\mathcal{O}(n k^2 \log^2 n)$
to $\mathcal{O}(n k^2 \log n)$.

Due to the fact that the $\mathcal{H}$-matrix-vector multiplications
appearing in the multiplication algorithm still require
$\mathcal{O}(n k^2 \log^2 n)$ operations, the new approach cannot improve
the asymptotic order of the \emph{entire} algorithm.
It can, however, significantly reduce the total runtime, since it
reduces the number of low-rank updates that are responsible for
a large part of the overall computational work.
Numerical experiments indicate that the new algorithm can reduce
the runtime by 50 percent or more, particularly for very large
matrices.

The article starts with a brief recollection of the structure of
$\mathcal{H}$-matrices in Section~2.
Section~3 describes the fundamental algorithms for the matrix-vector
multiplication and low-rank approximation and provides us with the
complexity estimates required for the analysis of the new algorithm.
Section~4 introduces a new algorithm for computing the
$\mathcal{H}$-matrix product using accumulated updates based on
the three basic operations ``addproduct'', that adds a product
to an accumulator, ``split'', that creates accumulators for
submatrices, and ``flush'', that adds the content of an accumulator
to an $\mathcal{H}$-matrix.
Section~5 is devoted to the analysis of the corresponding
computational work, in particular to the proof of an estimate
for the number of operations that shows that the rank-revealing
factorizations require only $\mathcal{O}(n k^2 \log n)$ operations
in the new algorithm compared to $\mathcal{O}(n k^2 \log^2 n)$ for
the standard approach.
Section~6 illustrates how accumulators can be incorporated into
higher-level operations like inversion or factorization.
Section~7 presents numerical experiments for boundary integral
operators that indicate that the new algorithm can significantly
reduce the runtime for the $\mathcal{H}$-LR and the
$\mathcal{H}$-Cholesky factorization.

%
%
\section{Hierarchical matrices}

Let $\Idx$ and $\Jdx$ be finite index sets.

In order to approximate a given matrix $G\in\bbbr^{\Idx\times\Jdx}$ by
a hierarchical matrix, we use a partition of the corresponding index
set $\Idx\times\Jdx$.
This partition is constructed based on hierarchical decompositions
of the index sets $\Idx$ and $\Jdx$.

%
%
\begin{definition}[Cluster tree]
\label{de:cluster_tree}
Let $\mathcal{T}$ be a labeled tree, and denote the label of a
node $t\in\mathcal{T}$ by $\hat t$.
We call $\mathcal{T}$ a \emph{cluster tree} for the index set $\Idx$
if
\begin{itemize}
  \item the root $r=\treeroot(\mathcal{T})$ is labeled with $\hat r=\Idx$,
  \item for $t\in\mathcal{T}$ with $\sons(t)\neq\emptyset$ we have
     \begin{equation*}
       \hat t = \bigcup_{t'\in\sons(t)} \hat t',\text{ and}
     \end{equation*}
  \item for $t\in\mathcal{T}$ and $t_1,t_2\in\sons(t)$ with
     $t_1\neq t_2$ we have $\hat t_1\cap\hat t_2=\emptyset$.
\end{itemize}
A cluster tree for $\Idx$ is usually denoted by $\ctI$, its nodes
are called \emph{clusters}, and its set of leaves is denoted by
\begin{equation*}
  \lfI := \{ t\in\ctI\ :\ \sons(t)=\emptyset \}.
\end{equation*}
\end{definition}

Let $\ctI$ and $\ctJ$ be cluster trees for $\Idx$ and $\Jdx$.
A pair $t\in\ctI$, $s\in\ctJ$ corresponds to a subset
$\hat t\times\hat s$ of $\Idx\times\Jdx$, i.e., to a submatrix
of $G$.
We organize these subsets in a tree.

%
%
\begin{definition}[Block tree]
\label{de:block_tree}
Let $\mathcal{T}$ be a labeled tree, and denote the label of a
node $b\in\mathcal{T}$ by $\hat b$.
We call $\mathcal{T}$ a \emph{block tree} for the cluster trees
$\ctI$ and $\ctJ$ if
\begin{itemize}
  \item for each node $b\in\mathcal{T}$ there are $t\in\ctI$ and
     $s\in\ctJ$ such that $b=(t,s)$,
  \item the root consists of the roots of $\ctI$ and $\ctJ$, i.e.,
     $r=\treeroot(\mathcal{T})$ has the form
     $r=(\treeroot(\ctI),\treeroot(\ctJ))$,
  \item for $b=(t,s)\in\mathcal{T}$ the label is given by
     $\hat b = \hat t\times\hat s$, and
  \item for $b=(t,s)\in\mathcal{T}$ with $\sons(b)\neq\emptyset$,
     we have $\sons(b)=\sons(t)\times\sons(s)$.
\end{itemize}
A block tree for $\ctI$ and $\ctJ$ is usually denoted by $\ctIJ$,
its nodes are called \emph{blocks}, and its set of leaves is denoted
by
\begin{equation*}
  \lfIJ := \{ b\in\ctIJ\ :\ \sons(b) = \emptyset \}.
\end{equation*}
For $b=(t,s)\in\ctIJ$, we call $t$ the \emph{row cluster} and
$s$ the \emph{column cluster}.
\end{definition}

Our definition implies that a block tree $\ctIJ$ is also a cluster
tree for the index set $\Idx\times\Jdx$.
The index sets corresponding to the leaves of a block tree $\ctIJ$
form a disjoint partition
\begin{equation*}
  \{ \hat b = \hat t\times\hat s \ :\ b=(t,s)\in\lfIJ \}
\end{equation*}
of the index set $\Idx\times\Jdx$, i.e., a matrix $G\in\bbbr^{\Idx\times\Jdx}$
is uniquely determined by its submatrices $G|_{\hat b}$ for all $b\in\lfIJ$.

Most algorithms for hierarchical matrices traverse the cluster
or block trees recursively.
In order to be able to derive rigorous complexity estimates for these
algorithms, we require a notation for subtrees.

%
%
\begin{definition}[Subtree]
For a cluster tree $\ctI$ and one of its clusters $t\in\ctI$,
we denote the subtree of $\ctI$ rooted in $t$ by $\ctsub{t}$.
It is a cluster tree for the index set $\hat t$, and we denote
its set of leaves by $\lfsub{t}$.

For a block tree $\ctIJ$ and one of its blocks $b=(t,s)\in\ctIJ$,
we denote the subtree of $\ctIJ$ rooted in $b$ by $\ctsub{b}$.
It is a block tree for the cluster trees $\ctsub{t}$ and $\ctsub{s}$,
and we denote its set of leaves by $\lfsub{b}$.
\end{definition}

Theoretically, a hierarchical matrix for a given block tree $\ctIJ$ can
be defined as a matrix such that $G|_{\hat b}$ has at most rank $k\in\bbbn_0$.
In practice, we have to take the representation of low-rank matrices
into account:
if the cardinalities $\#\hat t$ and $\#\hat s$ are larger than $k$, a
low-rank matrix can be efficiently represented in factorized form
\begin{align*}
  G|_{\hat b} &= A_b B_b^* &
  &\text{ with } A_b\in\bbbr^{\hat t\times k},\ B_b\in\bbbr^{\hat s\times k},
\end{align*}
since this representation requires only $(\#\hat t+\#\hat s) k$ units
of storage.
For small matrices, however, it is usually far more efficient to store
$G|_{\hat b}$ as a standard two-dimensional array.

To represent the different ways submatrices are handled, we split the
set of leaves $\lfIJ$ into the \emph{admissible leaves} $\lfaIJ$ that
are represented in factorized form and the \emph{inadmissible leaves}
$\lfiIJ$ that are represented in standard form.

%
%
\begin{definition}[Hierarchical matrix]
Let $G\in\bbbr^{\Idx\times\Jdx}$, and let $\ctIJ$ be a block tree for
$\ctI$ and $\ctJ$ with the sets $\lfaIJ$ and $\lfiIJ$ of admissible
and inadmissible leaves.
Let $k\in\bbbn_0$.

We call $G$ a \emph{hierarchical matrix} (or $\mathcal{H}$-matrix) of
local rank $k$ if for each admissible leaf $b=(t,s)\in\lfaIJ$ there are
$A_b\in\bbbr^{\hat t\times k}$ and $B_b\in\bbbr^{\hat s\times k}$ such that
\begin{equation}\label{eq:lowrank}
  G|_{\hat t\times\hat s} = A_b B_b^*.
\end{equation}
Together with the \emph{nearfield matrices} given by
$N_b:=G|_{\hat t\times\hat s}$ for each inadmissible leaf $b=(t,s)\in\lfiIJ$,
the matrix $G$ is uniquely determined by its \emph{hierarchical matrix
  representation}, the triple
$((A_b)_{b\in\lfaIJ}, (B_b)_{b\in\lfaIJ}, (N_b)_{b\in\lfiIJ})$.

The set of all hierarchical matrices for the block tree $\ctIJ$
and the local rank $k$ is denoted by $\mathcal{H}(\ctIJ,k)$.
\end{definition}

In typical applications, hierarchical matrix representations require
$\mathcal{O}(n k \log n)$ units of storage \cite{GRHA02,BEHA03,BO07,FAMEPR13}.

%
%
\section{Basic arithmetic operations}

If the block tree is constructed by standard algorithms \cite{GRHA02},
stiffness matrices corresponding to the discretization of a partial
differential operator are hierarchical matrices of local rank zero,
while integral operators can be approximated by hierarchical matrices
of low rank \cite{BE00a,BOGR02,BOGR04,BOCH14}.

In order to obtain an efficient preconditioner, we approximate the
inverse \cite{GRHA02,HA15} or the LR or Cholesky factorization
\cite[Section~7.6]{HA15} of a hierarchical matrix.
This task is typically handled by using rank-truncated arithmetic
operations \cite{HA99,GRHA02}.
For partial differential operators, domain-decomposition clustering
strategies have been demonstrated to significantly improve the
performance of hierarchical matrix preconditioners \cite{GRKRLE05a,GRKRLE05b},
since they lead to a large number of submatrices of rank zero.

We briefly recall four fundamental algorithms:
multiplying an $\mathcal{H}$-matrix by one or multiple vectors,
approximately adding low-rank matrices, approximately merging
low-rank block matrices to form larger low-rank matrices, and
approximately adding a low-rank matrix to an $\mathcal{H}$-matrix.

\paragraph{Matrix-vector multiplication.}
Let $G$ be a hierarchical matrix, $b=(t,s)\in\ctIJ$, $\alpha\in\bbbr$,
and let arbitrary matrices $X\in\bbbr^{\hat s\times\Kdx}$
and $Y\in\bbbr^{\hat t\times\Kdx}$ be given, where $\Kdx$ is an
arbitrary index set.
We are interested in performing the operations
\begin{align*}
  Y &\gets Y + \alpha G|_{\hat t\times\hat s} X, &
  X &\gets X + \alpha G|_{\hat t\times\hat s}^* Y.
\end{align*}
If $b$ is an inadmissible leaf, i.e., if $b=(t,s)\in\lfiIJ$ holds,
we have the nearfield matrix $N_b=G|_{\hat t\times\hat s}$ at our disposal
and can use the standard matrix multiplication.

If $b$ is an admissible leaf, i.e., if $b=(t,s)\in\lfaIJ$ holds,
we have $G|_{\hat t\times\hat s} = A_b B_b^*$ and can first compute
$\widehat{Z} := \alpha B_b^* X$ and then update
$Y \gets Y + A_b \widehat{Z}$ for the first operation or use
$\widehat{Z} := \alpha A_b^* Y$ and
$X \gets X + B_b \widehat{Z}$ for the second operation.

If $b$ is not a leaf, we consider all its sons $b'=(t',s')\in\sons(b)$
and perform the updates for the matrices $G|_{\hat t'\times\hat s'}$ and
the submatrices $X|_{\hat s'\times\Kdx}$ and $Y|_{\hat t'\times\Kdx}$ recursively.
Both algorithms are summarized in Figure~\ref{fi:addeval}.

%
%
\begin{figure}[t]
  \begin{quotation}
    \begin{tabbing}
      \textbf{procedure} addeval($t$, $s$, $\alpha$, $G$, $X$,
                              \textbf{var} $Y$);\\
      \textbf{if} $b=(t,s)\in\lfiIJ$ \textbf{then}\\
      \quad\= $Y \gets Y + \alpha N_b X$\\
      \textbf{else if} $b=(t,s)\in\lfaIJ$ \textbf{then begin}\\
      \> $\widehat{Z} \gets \alpha B_b^* X$;
      \quad $Y \gets Y + A_b \widehat{Z}$\\
      \textbf{end else}\\
      \> \textbf{for} $b'=(t',s')\in\sons(b)$ \textbf{do}\\
      \> \quad\= addeval($t'$, $s'$, $\alpha$, $G$,
            $X|_{\hat s\times\Kdx}$, $Y|_{\hat t\times\Kdx}$)\\
      \textbf{end}\\
      \\
      \textbf{procedure} addevaltrans($t$, $s$, $\alpha$, $G$, $Y$,
                                   \textbf{var} $X$);\\
      \textbf{if} $b=(t,s)\in\lfiIJ$ \textbf{then}\\
      \quad\= $X \gets Y + \alpha N_b^* X$\\
      \textbf{else if} $b=(t,s)\in\lfaIJ$ \textbf{then begin}\\
      \> $\widehat{Z} \gets \alpha A_b^* Y$;
      \quad $X \gets X + B_b \widehat{Z}$\\
      \textbf{end else}\\
      \> \textbf{for} $b'=(t',s')\in\sons(b)$ \textbf{do}\\
      \> \quad\= addevaltrans($t'$, $s'$, $\alpha$, $G$,
            $Y|_{\hat t\times\Kdx}$, $X|_{\hat s\times\Kdx}$)\\
      \textbf{end}
    \end{tabbing}
  \end{quotation}

  \caption{Multiplication $Y \gets Y + \alpha G|_{\hat t\times\hat s} X$
           or $X \gets X + \alpha G|_{\hat t\times\hat s}^* Y$}
  \label{fi:addeval}
\end{figure}

\paragraph{Truncation.}
Let $b=(t,s)\in\ctIJ$, and let $R\in\bbbr^{\hat t\times\hat s}$ be a
matrix of rank at most $\ell\leq\min\{\#\hat t,\#\hat s\}$.
Assume that $R$ is given in factorized form
\begin{align*}
  R &= A B^*, &
  A &\in \bbbr^{\hat t\times\ell},\ B\in\bbbr^{\hat s\times\ell},
\end{align*}
and let $k\in[0:\ell]$.
Our goal is to find the best rank-$k$ approximation of $R$.
We can take advantage of the factorized representation to efficiently
obtain a thin singular value decomposition of $R$:
let
\begin{equation*}
  B = Q_B R_B
\end{equation*}
be a thin QR factorization of $B$ with an orthogonal matrix
$Q_B\in\bbbr^{\hat s\times\ell}$ and an upper triangular matrix
$R_B\in\bbbr^{\ell\times\ell}$.
We introduce the matrix
\begin{equation*}
  \widehat{A} := A R_B^* \in\bbbr^{\hat t\times\ell}
\end{equation*}
and compute its thin singular value decomposition
\begin{equation*}
  \widehat{A} = U \Sigma \widehat{V}^*
\end{equation*}
with orthgonal matrices $U\in\bbbr^{\hat t\times\ell}$ and
$\widehat{V}\in\bbbr^{\ell\times\ell}$ and
\begin{align*}
  \Sigma &= \begin{pmatrix}
    \sigma_1 & & \\
    & \ddots & \\
    & & \sigma_\ell
  \end{pmatrix}, &
  \sigma_1 &\geq \sigma_2 \geq \ldots \geq \sigma_\ell \geq 0.
\end{align*}
A thin SVD of the original matrix $R$ is given by
\begin{align*}
  R &= A B^*
   = A R_B^* Q_B^*
   = \widehat{A} Q_B^*
   = U \Sigma \widehat{V}^* Q_B^*
   = U \Sigma (Q_B \widehat{V})^*
   = U \Sigma V^*
\end{align*}
with $V := Q_B \widehat{V}$.
The best rank-$k$ approximation with respect to the spectral
and the Frobenius norm is obtained by replacing the
smallest singular values $\sigma_{k+1},\ldots,\sigma_\ell$ in
$\Sigma$ by zero.

\paragraph{Truncated addition.}
Let $b=(t,s)\in\ctIJ$, and let $R_1,R_2\in\bbbr^{\hat t\times\hat s}$
be matrices of ranks at most $k_1,k_2\leq\min\{\#\hat t,\#\hat s\}$,
respectively.
Assume that these matrices are given in factorized form
\begin{align*}
  R_1 &= A_1 B_1^*, &
  A_1 &\in \bbbr^{\hat t\times k_1},\ B_1\in\bbbr^{\hat s\times k_1},\\
  R_2 &= A_2 B_2^*, &
  A_2 &\in \bbbr^{\hat t\times k_2},\ B_2\in\bbbr^{\hat s\times k_2},
\end{align*}
and let $\ell:=k_1+k_2$ and $k\in[0:\ell]$.
Our goal is to find the best rank-$k$ approximation of the
sum $R := R_1 + R_2$.
Due to
\begin{equation*}
  R = R_1 + R_2
  = A_1 B_1^* + A_2 B_2^*
  = \begin{pmatrix}
      A_1 & A_2
    \end{pmatrix}
    \begin{pmatrix}
      B_1 & B_2
    \end{pmatrix}^*,
\end{equation*}
this task reduces to computing the best rank-$k$ approximation of
a rank-$\ell$ matrix in factorized representation, and we have already
seen that we can use a thin SVD to obtain the solution.
The resulting algorithm is summarized in Figure~\ref{fi:rkadd}.

%
%
\begin{figure}
  \begin{quotation}
    \begin{tabbing}
      \textbf{procedure} rkadd($\alpha$, $R_1$, \textbf{var} $R_2$);\\
      \{ $R_1 = A B^*$, $R_2 = C D^*$ \}\\
      Find thin QR factorization
        $Q_B R_B = \begin{pmatrix} B & D \end{pmatrix}$;\\
      $\widehat{A} \gets \begin{pmatrix} \alpha A & C \end{pmatrix} R_B^*$;\\
      Find thin singular value decomposition
        $U \Sigma \widehat{V}^* = \widehat{A}$;\\
      Choose new rank $k$;\\
      $\widetilde{\Sigma} \gets \Sigma(1:k,:)$;\\
      $C \gets U(:,1:k)$;
      \quad $D \gets Q_B \widehat{V} \widetilde{\Sigma}^*$\\
      \textbf{end}
    \end{tabbing}
  \end{quotation}

  \caption{Truncated addition $R_2 \gets \trunc(R_2 + \alpha R_1)$}
  \label{fi:rkadd}
\end{figure}

\paragraph{Low-rank update.}
During the course of the standard $\mathcal{H}$-matrix multiplication
algorithm, we frequently have to add a low-rank matrix
$R = A B^*$ with $A\in\bbbr^{\hat t\times\Kdx}$,
$B\in\bbbr^{\hat s\times\Kdx}$ and $(t,s)\in\ctIJ$ to an
$\mathcal{H}$-submatrix $G|_{\hat t\times\hat s}$.
For any subsets $\hat t'\subseteq\hat t$ and $\hat s'\subseteq\hat s$,
we have
\begin{equation*}
  R|_{\hat t'\times\hat s'}
  = A|_{\hat t'\times\Kdx} B|_{\hat s'\times\Kdx}^*,
\end{equation*}
so any submatrix of the low-rank matrix $R$ is again a low-rank matrix,
and a factorized representation of $R$ gives rise to a factorized
representation of the submatrix without additional arithmetic
operations.
This leads to the simple recursive algorithm summarized in
Figure~\ref{fi:rkupdate} for approximately adding a low-rank matrix to an
$\mathcal{H}$-submatrix.

%
%
\begin{figure}
  \begin{quotation}
    \begin{tabbing}
      \textbf{procedure} rkupdate($t$, $s$, $\alpha$, $R$, \textbf{var} $G$);\\
      \{ $R = A B^*$ \}\\
      \textbf{if} $b=(t,s)\in\lfiIJ$ \textbf{then}\\
      \quad\= $N_b \gets N_b + A B^*$\\
      \textbf{else if} $b=(t,s)\in\lfaIJ$ \textbf{then}\\
      \> \{ $G|_{\hat t\times\hat s} = R'$ \}\\
      \> rkadd($\alpha$, $R$, $R'$)\\
      \textbf{else}\\
      \> \textbf{for} $b'=(t',s')\in\sons(b)$ \textbf{do}\\
      \> \quad\= rkupdate($t'$, $s'$, $\alpha$, $R|_{\hat t'\times s'}$, $G$)\\
      \textbf{end}
    \end{tabbing}
  \end{quotation}

  \caption{Truncated update $G|_{\hat t\times\hat s} \gets
  \blocktrunc(G|_{\hat t\times\hat s} + R)$}
  \label{fi:rkupdate}
\end{figure}

\paragraph{Splitting and merging.}
In order to be able to handle general block trees, it is convenient
to be able to split a low-rank matrix into submatrices and merge low-rank
submatrices into a larger low-rank submatrix.

Splitting a low-rank matrix is straightforward:
if $b=(t,s)\in\lfaIJ$ is an admissible leaf, we have
$G|_{\hat t\times\hat s} = A_b B_b^*$ and
$G|_{\hat t'\times\hat s'} = A_b|_{\hat{t}'\times k} B_b|_{\hat{s}'\times k}^*$
for all $t'\in\sons(t)$ and $s'\in\sons(s)$, i.e., we immediately
find factorized low-rank representations for submatrices.

Merging submatrices directly would typically lead to an increased rank,
so we once again apply truncation:
if we have $R_1=A_1 B_1^*$ and $R_2=A_2 B_2^*$ with
$A_1,A_2\in\bbbr^{\hat t\times k}$, $B_1\in\bbbr^{\hat s_1\times k}$
and $B_2\in\bbbr^{\hat s_2\times k}$, we can again use thin QR
factorizations
\begin{align*}
  B_1 &= Q_1 R_1, &
  B_2 &= Q_2 R_2
\end{align*}
with $R_1,R_2\in\bbbr^{k\times k}$ to find
\begin{align*}
  \begin{pmatrix}
    R_1 & R_2
  \end{pmatrix}
  &= \begin{pmatrix}
    A_1 B_1^* & A_2 B_2^*
  \end{pmatrix}
  = \begin{pmatrix}
    A_1 R_1^* Q_1^* & A_2 R_2^* Q_2^*
  \end{pmatrix}\\
  &= \underbrace{\begin{pmatrix}
    A_1 R_1^* & A_2 R_2^*
  \end{pmatrix}}_{=:\widehat{A}}
  \underbrace{\begin{pmatrix}
    Q_1^* & \\
    & Q_2^*
  \end{pmatrix}}_{=:\widehat{Q}}.
\end{align*}
The matrix $\widehat{A}$ has only $2k$ columns, so we can compute
its singular value decomposition efficiently, and multiplying the
resulting right singular vectors by $\widehat{Q}$ yields the
singular value decomposition of the block matrix.
We can proceed as in the algorithm ``rkadd'' to obtain a low-rank
approximation.

%
%
\begin{figure}
  \begin{quotation}
    \begin{tabbing}
      \textbf{procedure} rowmerge($R_1,\ldots,R_p$, \textbf{var} $R$);\\
      \{ $R_j = A_j B_j^*$, $R = A B^*$ \}\\
      \textbf{for} $j=1$ \textbf{to} $p$ \textbf{do}\\
      \quad\= Find thin QR factorization $B_j = Q_{B,j} R_{B,j}$;\\
      $\widehat{A} = \begin{pmatrix} A_1 R_{B,1}^* & \cdots & A_p R_{B,p}^*
                     \end{pmatrix}$;\\
      Find thin singular value decomposition
        $U \Sigma \widehat{V}^* = \widehat{A}$;\\
      Choose new rank $k$;\\
      $\widetilde{\Sigma} \gets \Sigma(1:k,:)$;\\
      $A \gets U(:,1:k)$;\quad
      $B \gets \begin{pmatrix} Q_{B,1} & & \\ & \ddots & \\ & & Q_{B,p}
               \end{pmatrix} \widehat{V} \widetilde{\Sigma}^*$\\
      \textbf{end}\\
      \\
      \textbf{procedure} rkmerge($(R_{ij})_{i\in[1:p],j\in[1:q]}$,
                           \textbf{var} $R$);\\
      \textbf{for} $j=1$ \textbf{to} $q$ \textbf{do}\\
      \quad\= rowmerge($R_{1j}^*,\ldots,R_{pj}^*$, $R_j$);\\
      rowmerge($R_1^*,\ldots,R_q^*$, $R$)\\
      \textbf{end}
    \end{tabbing}
  \end{quotation}

  \caption{Merging low-rank matrices}
  \label{fi:rkmerge}
\end{figure}

Applying this procedure to adjoint matrices (simply using $B A^*$
instead of $A B^*$), we can also merge block columns.
Merging first columns and then rows lead to the algorithm ``rkmerge''
summarized in Figure~\ref{fi:rkmerge}.

\paragraph{Complexity.}
Now let us consider the complexity of the basic algorithms
introduced so far.
We make the following standard assumptions:
\begin{subequations}
\begin{itemize}
  \item finding and applying a Householder projection in $\bbbr^n$
    takes not more than $\Cqr n$ operations, where $\Cqr$ is an
    absolute constant.
    This implies that the thin QR factorization of a matrix
    $X\in\bbbr^{n\times m}$ can be computed in $\Cqr n m \min\{n,m\}$
    operations and that applying the factor $Q$ to a matrix
    $Y\in\bbbr^{n\times\ell}$ takes not more than $\Cqr n \ell \min\{n,m\}$
    operations.
  \item the thin singular value decomposition of a matrix
    $X\in\bbbr^{n\times m}$ can be computed (up to machine accuracy) in
    not more than $\Csv n m \min\{n,m\}$ operations, where $\Csv$ is
    again an absolute constant.
  \item the block tree is \emph{admissible}, i.e., for all inadmissible
    leaves $b=(t,s)\in\lfiIJ$, the row cluster $t$ or the column
    cluster $s$ are leaves, so we have
    \begin{equation}\label{eq:blocktree_admissible}
      (t,s)\in\lfiIJ \Rightarrow t\in\lfI \vee s\in\lfJ
    \end{equation}
    for all $b=(t,s)\in\ctIJ$.
  \item the block tree is \emph{sparse} \cite{GRHA02}, i.e., there is a constant
    $\Csp\in\bbbr_{>0}$ such that
    \begin{align}
      \#\{ s\in\ctJ\ :\ (t,s)\in\ctIJ \} &\leq \Csp
     \label{eq:blocktree_rowsparse}
    \intertext{for all $t\in\ctI$ and}
      \#\{ t\in\ctI\ :\ (t,s)\in\ctIJ \} &\leq \Csp
     \label{eq:blocktree_colsparse}
    \end{align}
    for all $s\in\ctJ$.
  \item there is an upper bound $\Csn$ for the number of a cluster's sons,
    i.e.,
    \begin{align}\label{eq:clustertree_sons}
      \#\sons(t) &\leq \Csn, &
      \#\sons(s) &\leq \Csn
    \end{align}
    for all $t\in\ctI$ and $s\in\ctJ$.
  \item all ranks are bounded by the constant $k\in\bbbn$, i.e.,
    in addition to (\ref{eq:lowrank}), we also have
    \begin{align}\label{eq:clustertree_resolution}
      \#\hat t &\leq k, &
      \#\hat s &\leq k
    \end{align}
    for all leaves $t\in\lfI$, $s\in\lfJ$.
\end{itemize}
We also introduce the short notation
    \begin{gather*}
      \pI := \max\{ \level(t)\ :\ t\in\ctI \},\\
      \pJ := \max\{ \level(s)\ :\ s\in\ctJ \},\\
      \pIJ := \max\{ \level(b)\ :\ b\in\ctIJ \}.
    \end{gather*}
for the \emph{depths} of the trees involved in our algorithms.
\end{subequations}
The combination of (\ref{eq:blocktree_admissible}) and
(\ref{eq:clustertree_resolution}) ensures that the ranks of
all submatrices $G|_{\hat t\times\hat s}$ corresponding to leaves
$b=(t,s)\in\lfIJ$ of the block tree are bounded by $k$.

We will apply the algorithms only to index sets $\Kdx$ satisfying
$\#\Kdx\leq k$, and we use this inequality to keep the following
estimates simple.

We first consider the algorithm ``addeval''.
If $b=(t,s)\in\lfiIJ$, we multiply $N_b$ directly by $X$.
This takes not more than
$(\#\hat t) (2 \#\hat s-1) (\#\Kdx)$ operations, and
adding the result to $Y$ takes $(\#\hat t) (\#\Kdx)$ operations, for
a total of $2 (\#\hat t) (\#\hat s) (\#\Kdx)$ operations.
Scaling by $\alpha$ can be applied either to $X$ or to the result,
leading to additional $\min\{\#\hat t,\#\hat s\} (\#\Kdx)$ operations.
Due to (\ref{eq:blocktree_admissible}), we have $t\in\lfI$
or $s\in\lfJ$, and due to (\ref{eq:clustertree_resolution}),
we find $\#\hat t\leq k$ or $\#\hat s\leq k$.
This yields the simple bound
\begin{equation*}
  2 k (\#\hat t + \#\hat s) (\#\Kdx)
  \leq 2 k^2 (\#\hat t + \#\hat s)
\end{equation*}
for the number of operations.

If $b=(t,s)\in\lfaIJ$, computing $\widehat{Z}$ takes
$k (2\#\hat s-1) (\#\Kdx)$ operations, and scaling the result
by $\alpha$ takes $k (\#\Kdx)$ operations.
Adding the product $A_b \widehat{Z}$ to $Y$ then takes $2 (\#\hat t) k (\#\Kdx)$
operations, for a total of
\begin{equation*}
  2 k (\#\hat t + \#\hat s) (\#\Kdx)
  \leq 2 k (\#\hat t + \#\hat s) (\#\Kdx)
  \leq 2 k^2 (\#\hat t + \#\hat s)
\end{equation*}
operations.
Due to the recursive structure of the algorithm, we find that
\begin{align*}
  \Wev(t,s)
  &:= \left\{\begin{array}{l}
        2 k^2 (\#\hat t + \#\hat s)
        \quad\text{ if } \sons(t,s)=\emptyset,\\
        \sum_{(t',s')\in\sons(b)} \Wev(t',s')
        \quad\text{ otherwise}
      \end{array}\right.
\end{align*}
for all $b=(t,s)\in\ctIJ$.
is a bound for the total number of operations.
Due to symmetry, we obtain a similar result for the algorithm
``addevaltrans''.

A straightforward induction yields
\begin{equation}\label{eq:Wev_sum}
  \Wev(t,s)
  \leq 2 k^2 \left(
      \sum_{(t',s')\in\ctsub{b}} \#\hat t'
      + \sum_{(t',s')\in\ctsub{b}} \#\hat s' \right)
\end{equation}
for all $b=(t,s)\in\ctIJ$.
Since the definition of the block tree implies
$\level(t),\level(s)\leq\level(b)$ for all blocks
$b=(t,s)\in\ctIJ$, we can use (\ref{eq:blocktree_rowsparse}) and the
fact that clusters on the same level are disjoint to find
\begin{subequations}
\begin{align}
  \sum_{b=(t,s)\in\ctIJ} \#\hat t
  &= \sum_{\substack{t\in\ctI\\\level(t)\leq\pIJ}}
     \sum_{\substack{s\in\ctJ\\(t,s)\in\ctIJ}} \#\hat t
   \leq \Csp \sum_{\substack{t\in\ctI\\\level(t)\leq\pIJ}} \#\hat t\notag\\
  &= \Csp \sum_{\ell=0}^{\pIJ} \sum_{\substack{t\in\ctI\\\level(t)=\ell}}
        \#\hat t
   = \Csp \sum_{\ell=0}^{\pIJ} \#\bigcup_{\substack{t\in\ctI\\\level(t)=\ell}}
        \hat t
   \leq \Csp (\pIJ+1) \#\Idx.\label{eq:block_t_sum}
\end{align}
Repeating the same argument with (\ref{eq:blocktree_colsparse}) yields
\begin{equation}\label{eq:block_s_sum}
  \sum_{b=(t,s)\in\ctIJ} \#\hat s
  \leq \Csp (\pIJ+1) \#\Jdx.
\end{equation}
\end{subequations}
Applying these estimates to the subtree $\ctsub{b}$ instead of $\ctIJ$
gives us the final estimate
\begin{equation}\label{eq:Wev_bound}
  \Wev(t,s) \leq 2 \Csp k^2 (\pIJ+1) (\#\hat t + \#\hat s)
\end{equation}
for all $b=(t,s)\in\ctIJ$.
Now let us take a look at the truncated addition algorithm
``rkadd''.
Let $k_1,k_2\in\bbbn_0$ denote the number of columns of $R_1$
and $R_2$.
We will only apply the algorithm with $k_1,k_2\leq k$, and we
use this property to keep the estimates simple.
By our assumption, the thin QR factorization requires not
more than $\Cqr (\#\hat s) (k_1+k_2)^2\leq 4 \Cqr k^2 \#\hat s$
operations.
Setting up $\widehat{A}$ takes $(\#\hat t) k_1 \leq k \#\hat t$
operations to scale $A$ and not more than $2 (\#\hat t) (k_1+k_2)^2
\leq 8 k^2 \#\hat t$ operations to multiply by $R_B^*$.
By our assumption, the thin singular value decomposition
requires not more than $\Csv (\#\hat t) (k_1+k_2)^2\leq
4 \Csv k^2 \#\hat t$ operations.
The new rank $k$ is bounded by $k_1+k_2\leq 2k$, so scaling $\widehat{V}$
takes not more than $(k_1+k_2)^2\leq 4 k^2$ operations and applying $Q_B$
to $\widehat{V}$ takes not more than $\Cqr (\#\hat s) (k_1+k_2)^2
\leq 4 \Cqr k^2 \#\hat s$.
The total number of operations is bounded by
\begin{align*}
  4 \Cqr k^2 \#\hat s
    &+ k \#\hat t + 8 k^2 \#\hat t
   + 4 \Csv k^2 \#\hat t + 4 k^2\\
   &+ 4 \Cqr k^2 \#\hat s
   \leq \Cad k^2 (\#\hat t + \#\hat s)
\end{align*}
with $\Cad := \max\{ 8 \Cqr, 4 \Csv + 9 \}$.

The algorithm ``rkmerge'' can be handled in the same way to
show that not more than
\begin{align*}
  \sum_{j=1}^q &\left(2 \Cqr k^2 \#\hat t
               + \Csv (\Csn k)^2 \#\hat s \right)\\
  &+ 2 \Cqr k^2 \#\hat s + \Csv (\Csn k)^2 \#\hat t
   \leq \Cmg k^2 (\#\hat t + \#\hat s)
\end{align*}
operations are required to merge low-rank submatrices
of $G|_{\hat t\times\hat s}$, where the constant is given by
$\Cmg := \max\{ 2 \Csn \Cqr + \Csn^2 \Csv, \Csn^3 \Csv + 2 \Cqr \}$.

The algorithm ``rkupdate'' applies ``rkadd'' in admissible leaves
and directly multiplies $A$ and $B^*$ in inadmissible leaves.
In the latter case, the row cluster $t\in\ctI$ or the column
cluster $s\in\ctJ$ has to be a leaf of the cluster tree
due to (\ref{eq:blocktree_admissible}) and we can bound the
number of operations by
\begin{equation*}
  2 (\#\hat t) (\#\hat s) (\#\Kdx)
  \leq 2 k (\#\hat t + \#\hat s) k
  \leq \Cad k^2 (\#\hat t + \#\hat s).
\end{equation*}
As in the case of ``addeval'', a straightforward induction yields
that the total number of operations is bounded by
\begin{equation*}
  \Wup(t,s)
  := \left\{\begin{array}{l}
       \Cad k^2 (\#\hat t + \#\hat s)
       \quad\text{ if } \sons(t,s)=\emptyset,\\
       \sum_{(t',s')\in\sons(t,s)} \Wup(t',s')
       \quad\text{ otherwise}
     \end{array}\right.
\end{equation*}
for all $b=(t,s)\in\ctIJ$.
We can proceed as before to find
\begin{equation}\label{eq:Wup_bound}
  \Wup(t,s) \leq \Cad \Csp k^2 (\pIJ+1) (\#\hat t + \#\hat s)
\end{equation}
for all $b=(t,s)\in\ctIJ$.

%
%
\section{Matrix multiplication with accumulated updates}

Let us now consider the multiplication of two $\mathcal{H}$-matrices.
This operation is central to the entire field of
$\mathcal{H}$-matrix arithmetics, since it allows us to approximate
the inverse, the LR or Cholesky factorization, and even matrix
functions.

Following the lead of the well-known BLAS package, we write the matrix
multiplication as an update
\begin{equation*}
  Z \gets \blocktrunc(Z + \alpha X Y),
\end{equation*}
where $X,Y,Z$ are $\mathcal{H}$-matrices for blocktrees $\ctIJ$,
$\ctJK$, and $\ctIK$ corresponding to cluster trees $\ctI$, $\ctJ$,
and $\ctK$, respectively, $\alpha\in\bbbr$ is a scaling factor, and
``$\blocktrunc$'' denotes a suitable blockwise truncation.
Given that $\mathcal{H}$-matrices are defined recursively, it is
straightforward to define the matrix multiplication recursively as well,
so we consider local updates
\begin{equation*}
  Z|_{\hat t\times\hat r}
  \gets \blocktrunc(Z|_{\hat t\times\hat r}
                    + \alpha X|_{\hat t\times\hat s} Y|_{\hat s\times\hat r})
\end{equation*}
with $(t,s)\in\ctIJ$ and $(s,r)\in\ctJK$.
The key to an efficient approximate $\mathcal{H}$-matrix multiplication
is to take advantage of the low-rank properties of the factors
$X|_{\hat t\times\hat s}$ and $Y|_{\hat s\times\hat r}$.

If $(s,r)\in\lfiJK$, our assumption (\ref{eq:blocktree_admissible})
yields that $s\in\lfJ$ or $r\in\lfK$ holds.
In the first case, we have $\#\hat s\leq k$ due to
(\ref{eq:clustertree_resolution}) and obtain a factorized low-rank
representation
\begin{equation*}
  X|_{\hat t\times\hat s} Y|_{\hat s\times\hat r}
  = (X|_{\hat t\times\hat s} I) N_{(s,r)}
  = \widehat{A} \widehat{B}^*
\end{equation*}
with $\widehat{A} := X|_{\hat t\times\hat s} I$ and $\widehat{B} := N_{(s,r)}$.
In the second case, we have $\#\hat r\leq k$ due to
(\ref{eq:clustertree_resolution}) and can simply use
\begin{equation*}
  X|_{\hat t\times\hat s} Y|_{\hat s\times\hat r}
  = (X|_{\hat t\times\hat s} N_{s,r}) I
  = \widehat{A} \widehat{B}^*
\end{equation*}
with $\widehat{A} := X|_{\hat t\times\hat s} N_{s,r}$ and
$\widehat{B} := I\in\bbbr^{\hat r\times\hat r}$.
In both cases, $\widehat{A}$ can be computed using the ``addeval''
algorithm, and the low-rank representation $\widehat{A} \widehat{B}^*$
of the product can be added to $Z|_{\hat t\times\hat r}$ using the
``rkupdate'' algorithm.

If $(s,r)\in\lfaJK$, we have
\begin{equation*}
  Y|_{\hat s\times\hat r} = A_{(s,r)} B_{(s,r)}^*
\end{equation*}
and find
\begin{equation*}
  X|_{\hat t\times\hat s} Y|_{\hat s\times\hat r}
  = X|_{\hat t\times\hat s} A_{(s,r)} B_{(s,r)}^*
  = \widehat{A} \widehat{B}^*
\end{equation*}
with $\widehat{A} := X|_{\hat t\times\hat s} A_{(s,r)}$ and
$\widehat{B} := B_{(s,r)}$.
Once again, the matrix $\widehat{A}$ can be computed using
``addeval''.

If $(t,s)\in\lfIJ$ holds, we can follow a similar approach to
obtain low-rank representations for the product
$X|_{\hat t\times\hat s} Y|_{\hat s\times\hat r}$, replacing
``addeval'' by ``addevaltrans''.

If $(t,s)\not\in\lfIJ$ and $(s,r)\not\in\lfJK$, the definition
of the block tree implies that $t$, $s$, and $r$ cannot be leaves
of the corresponding cluster trees.
In this case, we split the product into
\begin{equation*}
  Z|_{\hat t'\times\hat r'}
  \gets \blocktrunc(Z|_{\hat t'\times\hat r'}
  + \alpha X|_{\hat t'\times\hat s'} Y|_{\hat s'\times\hat r'})
\end{equation*}
for all $t'\in\sons(t)$, $s'\in\sons(s)$, $r'\in\sons(r)$
and handle these updates by recursion.

If $(t,r)\in\lfIK$ or even $(t,r)\not\in\ctIK$, the blocks $(t',r')$
required by the recursion are not contained in the block tree $\ctIK$.
In this case, we create these sub-blocks temporarily, carry
out the recursion, and use the algorithm ``rkmerge'' to
merge the results into a new low-rank matrix if necessary.

The standard version of the multiplication algorithm constructs
the low-rank matrices $\widehat{A} \widehat{B}^*$ and directly
adds them to the corresponding submatrix of $Z$ using ``rkupdate''.
This approach can involve a significant number of operations:
entire subtrees of $\ctIK$ have to be traversed, and each of
the admissible leaves requires us to compute a QR factorization
and a singular value decomposition.

\paragraph{Accumulated updates.}
In order to reduce the computational work, we can use a variation
of the algorithm that is inspired by the \emph{matrix backward
transformation} for $\mathcal{H}^2$-matrices \cite{BO04a}:
instead of directly adding the low-rank matrices to the
$\mathcal{H}$-matrix, we accumulate them in auxiliary
low-rank matrices $\widehat{R}_{t,r}$ associated with all blocks
$(t,r)\in\ctIK$.
After all products have been treated, these low-rank matrices
can be ``flushed'' to the leaves of the final result:
starting with the root of $\ctIK$, for each block
$(t,r)\in\ctIK\setminus\lfIK$, the matrices $\widehat{R}_{t,r}$
are split into submatrices and added to $\widehat{R}_{t',r'}$
for all $t'\in\sons(t)$ and $r'\in\sons(r)$.

This approach ensures that each block $(t,r)\in\ctIK$ is propagated
only once to its sons and that each leaf $(t,r)\in\lfIK$ is only
updated once, so the number of low-rank updates can be significantly
reduced.

Storing the matrices $\widehat{R}_{t,r}$ for all blocks
$(t,r)\in\ctIK$ would significantly increase the storage requirements
of the algorithm.
Fortunately, we can avoid this disadvantage by rearranging the
arithmetic operations:
for each $(t,r)\in\ctIK$, we define an \emph{accumulator}
consisting of the matrix $\widehat{R}_{t,r}$ and a set
$P_{t,r}$ of triples
$(\alpha,s,X,Y)\in\bbbr\times\ctJ
                       \times\mathcal{H}(\ctsub{(t,s)},k)
                       \times\mathcal{H}(\ctsub{(s,r)},k)$
of pending products $\alpha X Y$.
A product is considered \emph{pending} if $(t,s)\in\ctIJ\setminus\lfIJ$
and $(s,r)\in\ctJK\setminus\lfJK$, i.e., if the product cannot be
immediately reduced to low-rank form but has to be treated in the
sons of $(t,r)$.

Apart from constructors and destructors, we define three
operations for accumulators:
\begin{itemize}
  \item the \emph{addproduct} operation adds a product
    $\alpha X Y$ with $\alpha\in\bbbr$, $X\in\mathcal{H}(\ctsub{(t,s)},k)$,
    $Y\in\mathcal{H}(\ctsub{(s,r)},k)$ to an accumulator for
    $(t,r)\in\ctIK$.
    If $(t,s)$ or $(s,r)$ is a leaf, the product is evaluated and
    added to $\widehat{R}_{t,r}$.
    Otherwise, it is added to the set $P_{t,r}$ or pending products.
  \item the \emph{split} operation takes an accumulator for
    $(t,r)\in\ctI\times\ctK$ and creates accumulators for the sons
    $(t',r')\in\sons(t)\times\sons(r)$ that inherit the already
    assembled matrix $\widehat{R}_{t,r}|_{\hat t'\times\hat r'}$.
    If $(\alpha,s,X,Y)\in P_{t,r}$ satisfies $(t',s')\in\lfIJ$
    or $(s',r')\in\lfJK$ for $s'\in\sons(s)$, the product is
    evaluated and its low-rank representation is added to
    $\widehat{R}_{t',r'}$.
    Otherwise $(\alpha,s',X|_{\hat t'\times\hat s'},Y|_{\hat s'\times\hat
      r'})$
    is added to the set $P_{t',r'}$ of pending products for the son.
  \item the \emph{flush} operation adds all products contained
    in an accumulator to an $\mathcal{H}$-matrix.
\end{itemize}
Instead of adding the product of $\mathcal{H}$-matrices to another
$\mathcal{H}$-matrix, we create an accumulator and use the
``add\-pro\-duct'' operation to turn handling the product over to it.
If we only want to compute the product, we can use the
``flush'' operation directly.
If we want to perform more complicated operations like inverting
a matrix, we can use the ``split'' operation to switch to
submatrices and defer flushing the accumulator until the results
are actually needed.

An efficient implementation of the ``flush'' operation can use
``split'' to shift the responsibility for the accumulated products
to the sons and then ``flush'' the sons' accumulators recursively.
If the sons' accumulators are deleted afterwards, the algorithm
only has to store accumulators for siblings along one branch of
the block tree at a time instead of for the entire block tree,
and the storage requirements can be significantly reduced.

%
%
\begin{figure}
  \begin{quotation}
    \begin{tabbing}
      \textbf{procedure} addproduct($\alpha$, $s$, $X$, $Y$,
                \textbf{var} $\widehat{R}_{t,r}$, $P_{t,r}$);\\
      \textbf{if} $(t,s)\in\lfiIJ$ \textbf{then begin}\\
      \quad\= \textbf{if} $\#\hat t \leq \#\hat s$ \textbf{then begin}\\
      \> \quad\= $\widehat{A} \gets I\in\bbbr^{\hat t\times\hat t}$;
           \quad $\widehat{B} \gets 0\in\bbbr^{\hat r\times\hat t}$;\\
      \> \> addevaltrans($s$, $r$, $1$, $Y$,
                  $N_{X,(t,s)}^*$, $\widehat{B}$)\\
      \> \textbf{end else begin}\\
      \> \> $\widehat{A} \gets N_{X,(t,s)}$;
           \quad $\widehat{B} \gets 0\in\bbbr^{\hat r\times\hat s}$;\\
      \> \> addevaltrans($s$, $r$, $1$, $Y$,
                  $I$, $\widehat{B}$)\\
      \> \textbf{end};\\
      \> rkadd($\alpha$, $\widehat{A} \widehat{B}^*$, $\widehat{R}_{t,r}$)\\
      \textbf{end else if} $(s,r)\in\lfiJK$ \textbf{then begin}\\
      \> \textbf{if} $\#\hat r \leq \#\hat s$ \textbf{then begin}\\
      \> \> $\widehat{B} \gets I\in\bbbr^{\hat r\times\hat r}$;
           \quad $\widehat{A} \gets 0\in\bbbr^{\hat t\times\hat r}$;\\
      \> \> addeval($t$, $s$, $\alpha$, $X$,
                  $N_{Y,(s,r)}$, $\widehat{A}$)\\
      \> \textbf{end else begin}\\
      \> \> $\widehat{B} \gets N_{Y,(s,r)}^*$;
           \quad $\widehat{A} \gets 0\in\bbbr^{\hat t\times\hat s}$;\\
      \> \> addeval($t$, $s$, $1$, $X$,
                  $I$, $\widehat{A}$)\\
      \> \textbf{end};\\
      \> rkadd($\alpha$, $\widehat{A} \widehat{B}^*$, $\widehat{R}_{t,r}$)\\
      \textbf{end else if} $(t,s)\in\lfaIJ$ \textbf{then begin}\\
      \> $\widehat{A} \gets A_{X,(t,s)}$;
           \quad $\widehat{B} \gets 0\in\bbbr^{\hat r\times k}$;\\
       \> addevaltrans($s$, $r$, $1$, $Y$,
                  $B_{X,(t,s)}$, $\widehat{B}$);\\
      \> rkadd($\alpha$, $\widehat{A} \widehat{B}^*$, $\widehat{R}_{t,r}$)\\
      \textbf{end else if} $(s,r)\in\lfaJK$ \textbf{then begin}\\
      \> $\widehat{B} \gets B_{Y,(s,r)}$;
           \quad $\widehat{A} \gets 0\in\bbbr^{\hat t\times k}$;\\
      \> addeval($t$, $s$, $1$, $X$,
                  $A_{Y,(s,r)}$, $\widehat{A}$);\\
      \> rkadd($\alpha$, $\widehat{A} \widehat{B}^*$, $\widehat{R}_{t,r}$)\\
      \textbf{end else}\\
      \> $P_{t,r} \gets P_{t,r} \cup \{ (\alpha,s,X,Y) \}$\\
      \textbf{end}
    \end{tabbing}
  \end{quotation}

  \caption{Adding a product to an accumulator $(\widehat{R}_{t,r},P_{t,r})$}
  \label{fi:addproduct}
\end{figure}

The ``flush'' operation can be formulated using the ``split''
operation, that in turn can be formulated using the ``addproduct''
operation.
The ``addproduct'' operation can be realized as described before:
if $(t,s)$ or $(s,r)$ are leaves, a factorized low-rank representation
of the product can be obtained using the ``addeval'' and
``addevaltrans'' algorithms.
If both $(t,s)$ and $(s,r)$ are not leaves, the product has to be
added to the list of pending products.
The resulting algorithm is given in Figure~\ref{fi:addproduct}.

%
%
\begin{figure}
  \begin{quotation}
    \begin{tabbing}
      \textbf{procedure} split($\widehat{R}_{t,r}$, $P_{t,r}$,
                \textbf{var}
                $(\widehat{R}_{t',r'},P_{t',r'})_{t',r'}$);\\
      \textbf{for} $t'\in\sons(t),\ r'\in\sons(r)$ \textbf{do begin}\\
      \quad\= $\widehat{R}_{t',r'}
                  \gets \widehat{R}_{t,r}|_{\hat{t}'\times\hat{r}'}$;\\
      \> $P_{t',r'} \gets \emptyset$;\\
      \> \textbf{for} $(\alpha,s,X,Y)\in P_{t,r}$ \textbf{do}\\
      \> \quad\= \textbf{for} $s'\in\sons(s)$ \textbf{do}\\
      \> \> \quad\= addproduct($\alpha$, $s'$,
             $X|_{\hat{t}'\times\hat{s}'}$, $Y|_{\hat{s}'\times\hat{r}'}$,
             $\widehat{R}_{t',r'}$, $P_{t',r'}$)\\
      \textbf{end}\\
    \end{tabbing}
  \end{quotation}

  \caption{Splitting an accumulator into accumulators for son blocks
           $(t',r')\in\sons(t)\times\sons(r)$}
  \label{fi:split}
\end{figure}

Using the ``addproduct'' algorithm, splitting an accumulator to
create accumulators for son blocks is straightforward:
if $\widehat{R}_{t,r} = A B^*$, we have
$\widehat{R}_{t,r}|_{\hat{t}'\times\hat{r}'} = A|_{\hat{t}'\times k}
B|_{\hat{r}'\times k}$ and can initialize the matrices $\widehat{R}_{t',r'}$
for the sons $t'\in\sons(t)$ and $r'\in\sons(r)$ accordingly.
The pending products $(\alpha,s,X,Y)\in P_{t,r}$ can be handled using
``addproduct'':
since $(t,s)$ and $(s,r)$ are not leaves, Definition~\ref{de:block_tree}
implies $\sons(s)\neq\emptyset$, so we can simply add the products
$\alpha X|_{\hat{t}'\times\hat{s}'} Y|_{\hat{s}'\times\hat{r}'}$ for all
$s'\in\sons(s)$ to either $\widehat{R}_{t',r'}$ or $P_{t',r'}$ using
``addproduct''.
The procedure is given in Figure~\ref{fi:split}.

%
%
\begin{figure}
  \begin{quotation}
    \begin{tabbing}
      \textbf{procedure} flush(\textbf{var} $\widehat{R}_{t,r}$, $P_{t,r}$,
                                            $Z$);\\
      \textbf{if} $P_{t,r} = \emptyset$ \textbf{do}\\
      \quad\= rkupdate($t$, $r$, $1$, $\widehat{R}_{t,r}$, $Z$)\\
      \textbf{else if} $\sons(t,r)\neq\emptyset$ \textbf{do begin}\\
      \> split($\widehat{R}_{t,r}$, $P_{t,r}$,
               $(\widehat{R}_{t',r'},P_{t',r'})_{t',r'}$);\\
      \> \textbf{for} $t'\in\sons(t)$, $r'\in\sons(r)$ \textbf{do}\\
      \> \quad\= flush($\widehat{R}_{t',r'}$, $P_{t',r'}$,
                       $Z|_{\hat{t}'\times\hat{s}'}$);\\
      \> Delete temporary accumulators $(\widehat{R}_{t',r'}$, $P_{t',r'})$\\
      \textbf{else begin}\\
      \> Create temporary matrices
         $\widetilde{Z}_{t',r'} \gets Z|_{\hat{t}'\times\hat{r}'}$\\
      \> \qquad for all $t'\in\sons(t)$, $r'\in\sons(r)$\\
      \> split($\widehat{R}_{t,r}$, $P_{t,r}$,
               $(\widehat{R}_{t',r'},P_{t',r'})_{t',r'}$);\\
      \> \textbf{for} $t'\in\sons(t)$, $r'\in\sons(r)$ \textbf{do}\\
      \> \quad\= flush($\widehat{R}_{t',r'}$, $P_{t',r'}$,
                       $\widetilde{Z}_{t',r'}$);\\
      \> Delete temporary accumulators $(\widehat{R}_{t',r'}$, $P_{t',r'})$\\
      \> \textbf{if} $Z$ is in standard representation \textbf{then}\\
      \> \> $Z|_{\hat{t}'\times\hat{r}'} \gets \widetilde{Z}_{t',r'}$
            for all $t'\in\sons(t)$, $r'\in\sons(r)$\\
      \> \textbf{else}\\
      \> \> rkmerge($(\widetilde{Z}_{t',r'})_{t',r'}$, $Z$);\\
      \> Delete temporary matrices $\widetilde{Z}_{t',r'}$\\
      \textbf{end}\\
      $\widehat{R}_{t,r} \gets 0$;
      \quad $P_{t,r} \gets \emptyset$
    \end{tabbing}
  \end{quotation}

  \caption{Flush an accumulator into an $\mathcal{H}$-matrix $Z$}
  \label{fi:flush}
\end{figure}

The ``flush'' operation can now be realized using the ``rkupdate''
algorithm if there are no more pending products, i.e., if
only the low-rank matrix $\widehat{R}_{t,r}$ contains information
that needs to be processed, and using the ``split'' algorithm
otherwise to move the contents of the accumulator to the sons
of the current block so that they can be handled by recursive
calls to ``flush''.
If there are pending products but $(t,r)$ has no sons, we split
$Z$ into temporary matrices $\widetilde{Z}_{t',r'}$ for
all $t'\in\sons(t)$ and $r'\in\sons(r)$ and proceed as before.
If $(t,r)$ is an inadmissible leaf or the descendant of an
inadmissible leaf, $Z$ and the matrices $\widetilde{Z}_{t',r'}$
are given in standard representation, so we can copy the
submatrices $\widetilde{Z}_{t',r'}$ directly back into $Z$.
Otherwise $Z$ is a low-rank matrix and we have to use the
``rkmerge'' algorithm to combine the low-rank matrices
$\widetilde{Z}_{t',r'}$ into the result.
The algorithm is summarized in Figure~\ref{fi:flush}.

%
%
\section{Complexity analysis}

If we use the algorithms ``addproduct'' and ``flush'' to compute
the approximated update
\begin{equation*}
  Z|_{\hat{t}\times\hat{r}}
  \gets \blocktrunc(Z|_{\hat{t}\times\hat{r}}
        + \alpha X|_{\hat{t}\times\hat{s}} Y|_{\hat{s}\times\hat{r}}),
\end{equation*}
most of the work takes place in the ``addproduct'' algorithm.
In fact, if we eliminate the first case in the ``flush'' algorithm
(cf. Figure~\ref{fi:flush}), we obtain an algorithm that performs
\emph{all} of its work in ``addproduct''.

For this reason, it makes sense to investigate how often ``addproduct''
is called during the ``flush'' algorithm and for which triplets $(t,s,r)$
these calls occur.
Since ``flush'' is a recursive algorithm, it makes sense to describe
its behaviour by a ``call tree'' that contains a node for each
call to ``addproduct''.
Since ``addproduct'' is only called for a triplet $(t,s,r)$ if
$(t,s)$ and $(s,r)$ are not leaves of $\ctIJ$ and $\ctJK$, respectively,
we arrive at the following structure.

%
%
\begin{definition}[Product tree]
Let $\ctIJ$ and $\ctJK$ be block trees for cluster trees
$\ctI$ and $\ctJ$, and $\ctJ$ and $\ctK$, respectively.

Let $\mathcal{T}$ be a labeled tree, and denote the label of
a node $c\in\mathcal{T}$ by $\hat{c}$.
We call $\mathcal{T}$ a \emph{product tree} for the block trees
$\ctIJ$ and $\ctJK$ if
\begin{itemize}
  \item for each note $c\in\mathcal{T}$ there are
    $t\in\ctI$, $s\in\ctJ$ and $r\in\ctK$ such that
    $c=(t,s,r)$,
  \item the root $r=\treeroot(\mathcal{T})$ of the product tree has the form
    $r=(\treeroot(\ctI),\treeroot(\ctJ),\treeroot(\ctK))$,
  \item for $c=(t,s,r)\in\mathcal{T}$ the label is given
    by $\hat{c}=\hat{t}\times\hat{s}\times\hat{r}$, and
  \item we have
    \begin{align*}
      \sons(c) &= \left\{\begin{array}{l}
        \emptyset \quad \text{ if } (t,s)\in\lfIJ \text{ or } (s,r)\in\lfJK,\\
        \sons(t)\times\sons(s)\times\sons(r) \quad \text{ otherwise}
      \end{array}\right.
    \end{align*}
    for all $c=(t,s,r)\in\mathcal{T}$.
\end{itemize}
A product tree for $\ctIJ$ and $\ctJK$ is usually denoted by
$\ctIJK$, its nodes are called \emph{products}.
\end{definition}

Let $\ctIJK$ be a product tree for the block trees $\ctIJ$
and $\ctJK$.

A simple induction yields
\begin{align}
  (t,s,r) \in\ctIJK
  &\iff (t,s)\in\ctIJ \wedge (s,r)\in\ctJK
   \label{eq:product_block}
\end{align}
for all $t\in\ctI$, $s\in\ctJ$, $r\in\ctK$
due to our Definition~\ref{de:block_tree} of block trees.

We can also see that $\ctIJK$ is a special cluster tree for the
index set $\Idx\times\Jdx\times\Kdx$.

If we call the procedure ``addproduct'' with the root triplet
$\treeroot(\ctIJK$) and use ``flush'', the algorithm ``addproduct'' will
be applied to all triplets $(t,s,r)\in\ctIJK$.
If $(t,s)$ or $(s,r)$ is a leaf, the algorithm uses
``addeval'' or ``addevaltrans'' to obtain a factorized low-rank
representation of the product $X|_{\hat{t}\times\hat{s}}
Y|_{\hat{s}\times\hat{r}}$ and ``rkadd'' to add it to the accumulator.
The first part takes either $\Wev(t,s)$ or $\Wev(s,r)$ operations,
the second part takes not more than $\Cad k^2 (\#\hat{t}+\#\hat{r})$
operations, for a total of
\begin{equation*}
  \Wev(t,s) + \Wev(s,r) + \Cad k^2 (\#\hat{t}+\#\hat{r})
\end{equation*}
operations.
If $(t,r)\not\in\ctIK\setminus\lfIK$ holds for $(t,s,r)\in\ctIJK$, we
also have to merge submatrices, and this takes not more than
$\Cmg k^2 (\#\hat{t}+\#\hat{r})$ operations.

%
%
\begin{theorem}[Complexity]
The new algorithm for the $\mathcal{H}$-matrix-mul\-ti\-pli\-cation with
accumulated updates (i.e., using ``addproduct'' for the root of
the product tree $\ctIJK$, followed by ``flush'') takes not more than
\begin{align*}
  \Cmm \Csp^2 &k^2 \max\{\pIJ+1,\pJK+1,\pIK+1\}^2
   (\#\Idx+\#\Jdx+\#\Kdx) \text{ operations},
\end{align*}
where $\Cmm := 3 \Cad + \Cmg + 2$.
\end{theorem}
\begin{proof}
Except for the final calls to ``rkupdate'', the number of operations
for ``flush'' can be bounded by
\begin{align*}
  \sum_{(t,s,r)\in\ctIJK}
  &\Wev(t,s) + \Wev(s,r)
   + (\Cad+\Cmg) k^2 (\#\hat{t}+\#\hat{r}).
\end{align*}
Since $(t,s,r)\in\ctIJK$ implies $(t,s)\in\ctIJ$, we can use
(\ref{eq:Wev_sum}) and (\ref{eq:blocktree_rowsparse}) to bound the
first term by
\begin{align*}
  \sum_{(t,s,r)\in\ctIJK}
    \Wev(t,s)
  &\leq 2 k^2 \sum_{(t,s,r)\in\ctIJK}
           \sum_{(t',s')\in\ctsub{(t,s)}} \#\hat{t}' + \#\hat{s}'\\
  &= 2 k^2 \sum_{(t,s)\in\ctIJ}
           \sum_{\substack{r\in\ctK\\ (s,r)\in\ctJK}}
           \sum_{(t',s')\in\ctsub{(t,s)}} \#\hat{t}' + \#\hat{s}'\\
  &\leq 2 \Csp k^2 \sum_{(t,s)\in\ctIJ}
           \sum_{(t',s')\in\ctsub{(t,s)}} \#\hat{t}' + \#\hat{s}'\\
  &= 2 \Csp k^2 \sum_{(t',s')\in\ctIJ}
           \sum_{\substack{(t,s)\ctIJ\\ (t',s')\in\ctsub{(t,s)}}}
               \#\hat{t}' + \#\hat{s}'.
\end{align*}
Since a block tree is a special cluster tree, the labels of
all blocks on a given level are disjoint.
This implies that any given block $(t',s')\in\ctIJ$ can have
at most one predecessor $(t,s)\in\ctIJ$ with $(t',s')\in\ctsub{(t,s)}$
on each level.
Since the number of levels is bounded by $\pIJ+1$, we find
\begin{align*}
  \sum_{(t,s,r)\in\ctIJK}
    \Wev(t,s)
  &\leq 2 \Csp k^2 (\pIJ+1) \sum_{(t',s')\in\ctIJ}
            \#\hat{t'}' + \#\hat{s}'
\end{align*}
and can use (\ref{eq:block_t_sum}) and (\ref{eq:block_s_sum}) to
conclude
\begin{align}
  \sum_{(t,s,r)\in\ctIJK}
    \Wev(t,s)
  &\leq 2 \Csp^2 k^2 (\pIJ+1)^2 (\#\Idx + \#\Jdx).
    \label{eq:complexity_mul_1}
\end{align}
For the second sum, we can follow the exact same approach, replacing
(\ref{eq:blocktree_rowsparse}) by (\ref{eq:blocktree_colsparse}), to
find the estimate
\begin{align}
  \sum_{(t,s,r)\in\ctIJK}
    \Wev(s,r)
  &\leq 2 k^2 \sum_{(s,r)\in\ctJK}
              \sum_{\substack{t\in\ctI\\ (t,s)\in\ctIJ}}
              \sum_{(s',r')\in\ctsub{(s,r)}}
                 \#\hat{s}' + \#\hat{r}'\notag\\
  &\leq 2 \Csp k^2 (\pJK+1) \sum_{(s',r')\in\ctJK}
                 \#\hat{s}' + \#\hat{r}'\notag\\
  &\leq 2 \Csp^2 k^2 (\pJK+1)^2 (\#\Jdx + \#\Kdx).
     \label{eq:complexity_mul_2}
\end{align}
For the third sum, we can again use (\ref{eq:blocktree_rowsparse}) and
(\ref{eq:blocktree_colsparse}), respectively, to get
\begin{align*}
  \sum_{(t,s,r)\in\ctIJK} \#\hat{t}
  &= \sum_{t\in\ctI} \sum_{\substack{s\in\ctJ\\(t,s)\in\ctIJ}}
      \sum_{\substack{r\in\ctK\\(s,r)\in\ctJK}}
      \#\hat{t}
   \leq \Csp^2 \sum_{\substack{t\in\ctI\\\level(t)\leq\pIJ}} \#\hat{t},\\
  \sum_{(t,s,r)\in\ctIJK} \#\hat{r}
  &= \sum_{r\in\ctK} \sum_{\substack{s\in\ctJ\\(s,r)\in\ctJK}}
      \sum_{\substack{t\in\ctI\\(t,s)\in\ctIJ}}
      \#\hat{s}
   \leq \Csp^2 \sum_{\substack{r\in\ctK\\\level(r)\leq\pJK}} \#\hat{r}.
\end{align*}
Since the clusters on the same level of a cluster tree are disjoint,
we have
\begin{align*}
  \sum_{\substack{t\in\ctI\\\level(t)\leq\pIJ}} \#\hat{t}
  &= \sum_{\ell=0}^{\pIJ} \sum_{\substack{t\in\ctI\\\level(t)=\ell}} \#\hat{t}
   \leq \sum_{\ell=0}^{\pIJ} \#\Idx
   = (\pIJ+1) \#\Idx,\\
  \sum_{\substack{r\in\ctK\\\level(r)\leq\pJK}} \#\hat{r}
  &= \sum_{\ell=0}^{\pJK} \sum_{\substack{r\in\ctK\\\level(r)=\ell}} \#\hat{r}
   \leq \sum_{\ell=0}^{\pJK} \#\Kdx
   = (\pJK+1) \#\Kdx
\end{align*}
and conclude that the third sum is bounded by
\begin{equation}\label{eq:complexity_mul_3}
  \Csp^2 (\Cad+\Cmg) k^2 \left( (\pIJ+1) \#\Idx + (\pJK+1) \#\Kdx \right).
\end{equation}
Now we have to consider the calls to ``rkupdate'' in the third line
of the ``flush'' algorithm.
If $(t,r)\in\ctIK$, we call ``rkadd'' for all leaves $(t',r')\in\ctIK$
descended from $(t,r)$, and this takes not more than
$\Cad k^2 (\#\hat t'+\#\hat r')$ operations.
Since every leaf appears at most once, we can use (\ref{eq:block_t_sum})
again to obtain
\begin{equation*}
  \sum_{(t',r')\in\ctIK} \Cad k^2 (\#\hat t' + \#\hat r')
  \leq \Cad \Csp k^2 (\pIK+1) (\#\Idx + \#\Kdx)
\end{equation*}
for the first case and
\begin{align*}
  \Cad \kern-10pt \sum_{(t,s,r)\in\ctIJK} \#\hat t+\hat r
  &\leq \Cad \Csp \left( \sum_{(t,s)\in\ctIJ} \#\hat t
  + \sum_{(s,r)\in\ctJK} \#\hat r \right)\\
  &\quad\leq \Cad \Csp^2 \max\{\pIJ+1,\pJK+1\} (\#\Idx + \#\Kdx)
\end{align*}
for the second.
Combining both estimates yields
\begin{equation}\label{eq:complexity_mul_4}
  2 \Cad \Csp^2 \max\{\pIJ+1,\pJK+1,\pIK+1\} (\#\Idx+\#\Kdx),
\end{equation}
and adding (\ref{eq:complexity_mul_1}), (\ref{eq:complexity_mul_2}),
and (\ref{eq:complexity_mul_3}) while using $\Cad \geq 1$ yields
the final result.
\end{proof}

%
%
\begin{remark}
Without accumulated updates, the call to ``rkadd'' in the
``addproduct'' algorithm would have to be replaced by a call to
``rkupdate''.

Since ``rkupdate'' has to traverse the entire subtree rooted
in $(t,r)$, avoiding it in favor of accumulated updates can significantly
reduce the overall work.
\end{remark}

%
%
\begin{remark}[Parallelization]
Since the ``flush'' operations for different sons of the same block
are independent, the new multiplication algorithm with accumulated
updates could be fairly attractive for parallel implementations
of $\mathcal{H}$-matrix arithmetic algorithms:
in a shared-memory system, updates to disjoint submatrices can
be carried out concurrently without the need for locking.
In a distributed-memory system, we can construct lists of submatrices
that have to be transmitted to other nodes during the course of
the ``addproduct'' algorithm and reduce communication to the
necessary minimum.
\end{remark}

%
%
\section{Inversion and factorization}

In most applications, the $\mathcal{H}$-matrix multiplication
is used to construct a preconditioner for a linear system, i.e.,
an approximation of the inverse of an $\mathcal{H}$-matrix.

When using accumulated updates, the corresponding algorithms
have to be slightly modified.
As a simple example, we consider the inversion \cite{GRHA02}.
More efficient algorithms like the $\mathcal{H}$-LR or the
$\mathcal{H}$-Cholesky factorization can be treated in a
similar way.

For the purposes of our example, we consider an $\mathcal{H}$-matrix
$G\in\mathcal{H}(\ctII,k)$ and assume that it and all of its principal
submatrices are invertible and that diagonal blocks $(t,t)\in\ctII$
with $t\in\ctI$ are not admissible.

To keep the presentation simple, we also assume that the cluster tree
$\ctI$ is a binary tree, i.e., that we have $\#\sons(t)=2$ for
all non-leaf clusters $t\in\ctI\setminus\lfI$.

We are interested in approximating the inverse of a submatrix
$\widehat{G} := G|_{\hat{t}\times\hat{t}}$ for $t\in\ctI$.
If $t$ is a leaf cluster, the block $(t,t)$ has to be an inadmissible
leaf of $\ctII$, so $\widehat{G}$ is stored as a dense
matrix in standard representation and we can compute its inverse
directly by standard linear algebra.

If $t$ is not a leaf cluster, we have $\sons(t)=\{t_1,t_2\}$ for
$t_1,t_2\in\ctI$ and $(t,t)\in\ctII\setminus\lfII$.
We split $\widehat{G}$ into
\begin{align*}
  \widehat{G}_{11} &:= G|_{\hat{t}_1\times\hat{t}_1}, &
  \widehat{G}_{12} &:= G|_{\hat{t}_1\times\hat{t}_2},\\
  \widehat{G}_{21} &:= G|_{\hat{t}_2\times\hat{t}_1}, &
  \widehat{G}_{22} &:= G|_{\hat{t}_2\times\hat{t}_2}
\end{align*}
and get
\begin{equation*}
  \widehat{G} = \begin{pmatrix}
    \widehat{G}_{11} & \widehat{G}_{12}\\
    \widehat{G}_{21} & \widehat{G}_{22}
  \end{pmatrix}.
\end{equation*}
Due to our assumptions, $\widehat{G}$, $\widehat{G}_{11}$ and
$\widehat{G}_{22}$ are invertible.

The standard algorithm for inverting an $\mathcal{H}$-matrix
can be derived by a block LR factorization:
we have
\begin{align*}
  \widehat{G}
   &= \begin{pmatrix}
       I & \\
       \widehat{G}_{21} \widehat{G}_{11}^{-1} & I
     \end{pmatrix}
     \begin{pmatrix}
       \widehat{G}_{11} & \widehat{G}_{12}\\
       & \widehat{G}_{22} - \widehat{G}_{21} \widehat{G}_{11}^{-1}
                           \widehat{G}_{12}
     \end{pmatrix}\\
   &= \begin{pmatrix}
       I & \\
       \widehat{G}_{21} \widehat{G}_{11}^{-1} & I
     \end{pmatrix}
     \begin{pmatrix}
       \widehat{G}_{11} & \widehat{G}_{12}\\
       & S
     \end{pmatrix}
\end{align*}
with the Schur complement $S = \widehat{G}_{22} - \widehat{G}_{21}
\widehat{G}_{11}^{-1} \widehat{G}_{12}$ that is invertible since
$\widehat{G}$ is invertible.
Inverting the block triangular matrices yields
\begin{align*}
  \widehat{G}^{-1}
  &= \begin{pmatrix}
       \widehat{G}_{11}^{-1} & -\widehat{G}_{11}^{-1} \widehat{G}_{12} S^{-1}\\
       & S^{-1}
     \end{pmatrix}
     \begin{pmatrix}
       I & \\
       -\widehat{G}_{21} \widehat{G}_{11}^{-1} & I
     \end{pmatrix}\\
  &= \begin{pmatrix}
       \widehat{G}_{11}^{-1}
       + \widehat{G}_{11}^{-1} \widehat{G}_{12} S^{-1}
         \widehat{G}_{21} \widehat{G}_{11}^{-1} &
       -\widehat{G}_{11}^{-1} \widehat{G}_{12} S^{-1}\\
       -S^{-1} \widehat{G}_{21} \widehat{G}_{11}^{-1} &
       S^{-1}
     \end{pmatrix}.
\end{align*}
We can see that only matrix multiplications and the inversion
of the submatrices $\widehat{G}_{11}$ and $S$ are required, and
the inversions can be handled by recursion.

The entire computation can be split into six steps:
\begin{enumerate}
  \item Invert $\widehat{G}_{11}$.
  \item Compute $H_{12} := \widehat{G}_{11}^{-1} \widehat{G}_{12}$ and
        $H_{21} := \widehat{G}_{21} \widehat{G}_{11}^{-1}$.
  \item Compute $S := \widehat{G}_{22} - H_{21} \widehat{G}_{12}$.
  \item Invert $S$.
  \item Compute
        $H_{12}' := -H_{12} S^{-1}$ and
        $H_{21}' := -S^{-1} H_{21}$.
  \item Compute
        $H_{11}' := \widehat{G}_{11}^{-1} - H_{12} H_{21}'$.
\end{enumerate}
After these steps, the inverse is given by
\begin{equation*}
  \widehat{G}^{-1}
  = \begin{pmatrix}
      H_{11}' & H_{12}'\\
      H_{21}' & S^{-1}
    \end{pmatrix},
\end{equation*}
and due to the algorithm's structure, we can directly overwrite
$\widehat{G}_{22}$ first by $S$ and then by $S^{-1}$, $\widehat{G}_{12}$
by $H_{12}'$, $\widehat{G}_{21}$ by $H_{21}'$, and $\widehat{G}_{11}$
first by $\widehat{G}_{11}^{-1}$ and then by $H_{11}'$.
We require additional storage for the auxiliary matrices $H_{12}$
and $H_{21}$.

In order to take advantage of accumulated updates, we represent
all updates applied so far to the matrix $\widehat{G}$ by an
accumulator.
In the course of the inversion, the accumulator is split into
accumulators for the submatrices $\widehat{G}_{11}$, $\widehat{G}_{12}$,
$\widehat{G}_{21}$, and $\widehat{G}_{22}$.

The first step is a simple recursive call.
The second step consists of using ``flush'' for the submatrices
$\widehat{G}_{12}$ and $\widehat{G}_{21}$, creating empty accumulators
for the auxiliary matrices $H_{21}$ and $H_{12}$ and using ``addproduct''
and ``flush'' to compute the required products.
In the third step, we simply use ``addproduct'' without ``flush'',
since the recursive call in the fourth step is able to handle
accumulators.
In the fifth step, we use ``addproduct'' and ``flush'' again to
compute $H_{12}'$ and $H_{21}'$.
The sixth step computes $H_{11}'$ in the same way by using
``addproduct'' and ``flush''.
The algorithm is summarized in Figure~\ref{fi:invert}.

%
%
\begin{figure}
  \begin{quotation}
    \begin{tabbing}
      \textbf{procedure} invert($t$, $\widehat{R}_{t,t}$, $P_{t,t}$,
                \textbf{var} $G$, $H$);\\
      \textbf{if} $\sons(t)=\emptyset$ \textbf{then}\\
      \quad\= $G|_{\hat{t}\times\hat{t}} \gets
                  G|_{\hat{t}\times\hat{t}}^{-1}$\\
      \textbf{else begin}\\
      \> split($\widehat{R}_{t,t}$, $P_{t,t}$,
               $(\widehat{R}_{t',s'},P_{t',s'})_{t',s'\in\sons(t)}$);\\
      \> invert($t_1$, $\widehat{R}_{t_1,t_1}$, $P_{t_1,t_1}$, $G$, $H$);\\
      \> flush($\widehat{R}_{t_1,t_2}$, $P_{t_1,t_2}$, $\widehat{G}_{12}$);\\
      \> flush($\widehat{R}_{t_2,t_1}$, $P_{t_2,t_1}$, $\widehat{G}_{21}$);\\
      \> $H_{12} \gets 0$;
         \quad $H_{21} \gets 0$;\\
      \> addproduct($1$, $t_1$, $\widehat{G}_{11}$, $\widehat{G}_{12}$,
                    $\widehat{R}_{t_1,t_2}$, $P_{t_1,t_2}$);\\
      \> flush($\widehat{R}_{t_1,t_2}$, $P_{t_1,t_2}$, $H_{12}$);\\
      \> addproduct($1$, $t_1$, $\widehat{G}_{21}$, $\widehat{G}_{11}$,
                    $\widehat{R}_{t_2,t_1}$, $P_{t_2,t_1}$);\\
      \> flush($\widehat{R}_{t_2,t_1}$, $P_{t_2,t_1}$, $H_{21}$);\\
      \> addproduct($-1$, $t_1$, $H_{21}$, $\widehat{G}_{12}$,
                    $\widehat{R}_{t_2,t_2}$, $P_{t_2,t_2}$);\\
      \> invert($t_2$, $\widehat{R}_{t_2,t_2}$, $P_{t_2,t_2}$, $G$, $H$);\\
      \> $\widehat{G}_{12} \gets 0$;
         \quad $\widehat{G}_{21} \gets 0$;\\
      \> addproduct($-1$, $t_2$, $H_{12}$, $\widehat{G}_{22}$,
                    $\widehat{R}_{t_1,t_2}$, $P_{t_1,t_2}$);\\
      \> flush($\widehat{R}_{t_1,t_2}$, $P_{t_1,t_2}$, $\widehat{G}_{12}$);\\
      \> addproduct($-1$, $t_2$, $\widehat{G}_{22}$, $H_{21}$,
                    $\widehat{R}_{t_2,t_1}$, $P_{t_2,t_1}$);\\
      \> flush($\widehat{R}_{t_2,t_1}$, $P_{t_2,t_1}$, $\widehat{G}_{21}$);\\
      \> addproduct($-1$, $t_2$, $H_{12}$, $\widehat{G}_{21}$,
                    $\widehat{R}_{t_2,t_2}$, $P_{t_2,t_2}$);\\
      \> flush($\widehat{R}_{t_2,t_2}$, $P_{t_2,t_2}$, $\widehat{G}_{11}$)\\
      \textbf{end}
    \end{tabbing}
  \end{quotation}

  \caption{Invert an $\mathcal{H}$-matrix $G$ using accumulated updates}
  \label{fi:invert}
\end{figure}

It is possible to prove that the computational work required to invert
an $\mathcal{H}$-matrix by the standard algorithm \emph{without} accumulated
updates is bounded by the computational work required to multiply the
matrix by itself.

If we use accumulated updates, the situation chan\-ges:
since ``flush'' is applied multiple times to the submatrices in the inversion
procedure, but only once to each submatrix in the multiplication procedure,
we cannot bound the computational work of the inversion by the work
for the multiplication with accumulated updates.
Fortunately, numerical experiments indicate that accumulating the
updates still significantly reduces the run-time of the inversion.
The same holds for the $\mathcal{H}$-LR and the $\mathcal{H}$-Cholesky
factorizations \cite[Section~7.6]{LI04,HA15}.

%
%
\section{Numerical experiments}

According to the theoretical estimates, we cannot expect the new
algorithm to lead to an improved \emph{order} of complexity, since
evaluating the products of all re\-le\-vant submatrices still requires
$\mathcal{O}(n k^2 p^2)$ operations.
But since the computationally intensive update and merge operations
require only $\mathcal{O}(n k^2 p)$ operations with accumulated
updates, we can hope that the new algorithm performs better in practice.

The following experiments were carried out using the
H2Lib\footnote{Open source, available at \url{http://www.h2lib.org}}
package.
The standard arithmetic operations are contained in its module
\texttt{harith}, while the operations with accumulated updates
are in \texttt{harith2}.
Both share the same functions for matrix-vector multiplications,
truncated updates, and merging.

We consider two $\mathcal{H}$-matrices:
the matrix $V$ is constructed by discretizing the single layer potential
operator on a polygonal approximation of the unit sphere using piecewise
constant basis functions.
The mesh is constructed by refining a double pyramid regularly and
projecting the resulting vertices to the unit sphere.
The $\mathcal{H}$-matrix approximation results from applying the
hybrid cross approximation (HCA) technique \cite{BOGR04} with an
interpolation order of $m=4$ and a cross approximation tolerance
of $10^{-5}$, followed by a simple truncation with a tolerance
of $10^{-4}$.

The second matrix $K$ is constructed by discretizing the double layer
potential operator (plus $1/2$ times the identity) using the same
procedure as for the matrix $V$.

%
%
\begin{figure}
  \begin{center}
    \includegraphics[width=0.45\textwidth]{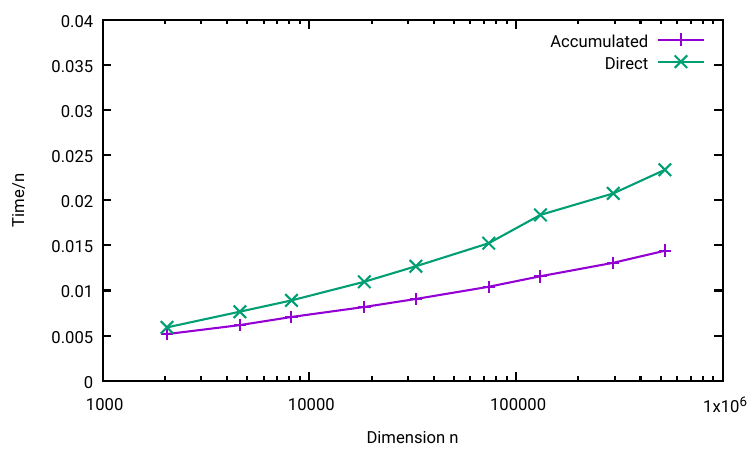}
    \includegraphics[width=0.45\textwidth]{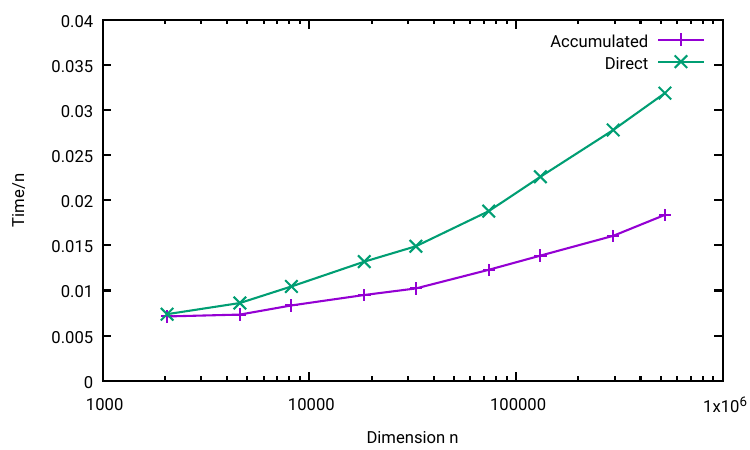}
  \end{center}

  \caption{Runtime per degree of freedom for the $\mathcal{H}$-matrix
           multiplication, single layer $V V$ on the left, double layer
           $K K$ on the right}
  \label{fi:mul}
\end{figure}

In a first experiment, we measure the runtime of the matrix
multiplication algorithms with a truncation tolerance of $10^{-4}$.
Figure~\ref{fi:mul} shows the runtime divided by the matrix
dimension $n$ using a logarithmic scale for the $n$ axis.
Both algorithms reach a relative accuracy well below $10^{-4}$ with
respect to the spectral norm, and we can see that the version with
accumulated updates has a significant advantage over the standard
direct approach, particularly for large matrices.
Although accumulating the updates requires additional truncation steps,
the measured total error is not significantly larger.
Since the new algorithm stores only one low-rank matrix for each ancestor
of the current block, the temporary storage requirements are negligible.

%
%
\begin{figure}
  \begin{center}
    \includegraphics[width=0.45\textwidth]{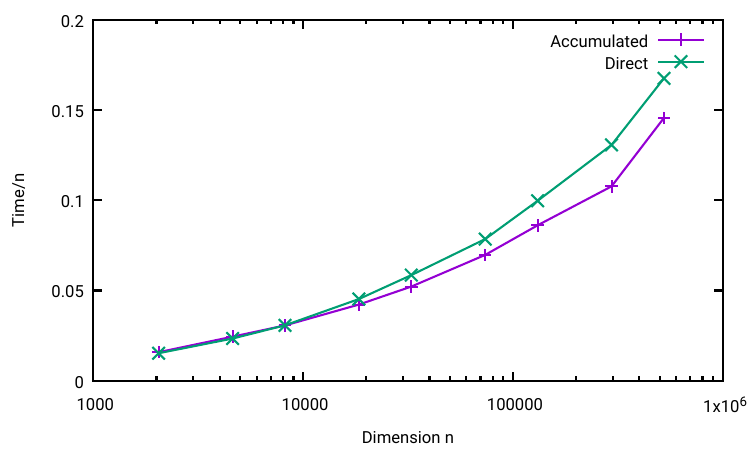}
    \includegraphics[width=0.45\textwidth]{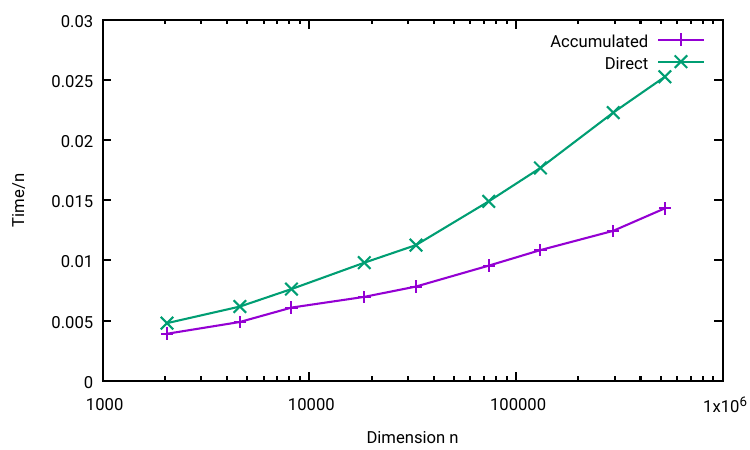}
  \end{center}

  \caption{Runtime per degree of freedom for the $\mathcal{H}$-matrix
           inversion, single layer $V^{-1}$ on the left, double layer
           $K^{-1}$ on the right}
  \label{fi:inv}
\end{figure}

In the next experiment, we consider the $\mathcal{H}$-matrix inversion,
again using a truncation tolerance of $10^{-4}$.
Since the inverse is frequently used as a preconditioner, we estimate
the spectral norm $\|I - \widetilde{G}^{-1} G\|_2$ using a power iteration.
For the single layer matrix $V$, this ``preconditioner error'' starts at
$6\times 10^{-4}$ for the smallest matrix and grows to $10^{-2}$ for the
largest, as is to be expected due to the increasing condition number.
For the double layer matrix $K$, the error lies between $2\times 10^{-2}$
and $1.1\times 10^{-1}$.
For $n=73\,728$, the error obtained by using accumulated updates is
almost four times larger than the one for the classical algorithm,
while for $n=524\,288$ both differ by only $13$ percent.
We can see in Figure~\ref{fi:inv} that accumulated updates again reduce
the runtime, but the effect is only very minor for the single layer matrix
and far more pronounced for the double layer matrix.

%
%
\begin{figure}
  \begin{center}
    \includegraphics[width=0.45\textwidth]{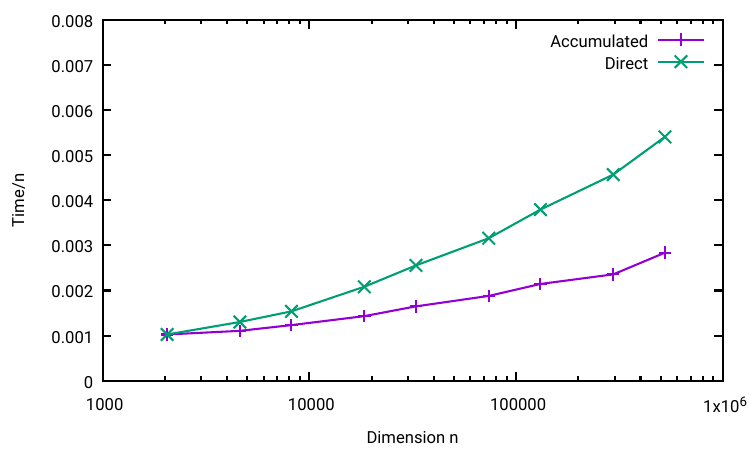}
    \includegraphics[width=0.45\textwidth]{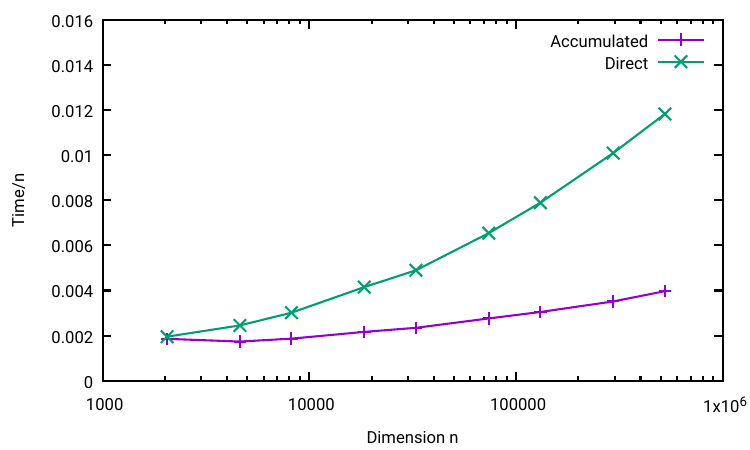}
  \end{center}

  \caption{Runtime per degree of freedom for the $\mathcal{H}$-matrix
           factorization, single layer $L L^* = V$ on the left, double layer
           $L R = K$ on the right}
  \label{fi:factor}
\end{figure}

In a final experiment, we investigate $\mathcal{H}$-matrix factorizations.
Since $V$ is symmetric and positive definite, we approximate its
Cholesky factorization $V \approx \widetilde{L} \widetilde{L}^*$, while
we use the standard LR factorization $K \approx \widetilde{L} \widetilde{R}$
for the matrix $K$.
The estimated preconditioner error
$\|I - (\widetilde{L} \widetilde{L}^*)^{-1} V\|_2$ for the
single layer matrix of dimension $n=524\,288$ is close to $2.2\times 10^{-3}$
for the algorithm with accumulated updates and close to $9.5\times 10^{-4}$
for the standard algorithm.
For the double layer matrix of the same dimension,
the estimated error $\|I - (\widetilde{L} \widetilde{R})^{-1} K\|_2$
is close to $5.6\times 10^{-1}$ for the new algorithm and close to
$3.8\times 10^{-1}$ for the standard algorithm.
Figure~\ref{fi:factor} shows that accumulated updates significantly
reduce the runtime for both factorizations.

In summary, accumulated updates reduce the runtime of
the $\mathcal{H}$-matrix multiplication and factorization by a
factor between two and three in our experiments while the error
is only moderately increased.
The same speed-up can be observed for the inversion of the double
layer matrix, while the improvement for the single layer matrix
is significantly smaller.

%
%

\bibliographystyle{plain}
\bibliography{hmatrix}

\end{document}